\documentclass[9pt, journal]{IEEEtran}

\usepackage{amssymb, bm, mathrsfs, color, subcaption,booktabs}

\newtheorem{theorem}{Theorem}[section]

\newtheorem{definition}[theorem]{Definition}

\newtheorem{assumption}[theorem]{Assumption}

\newtheorem{lemma}[theorem]{Lemma}
\newtheorem{remark}[theorem]{Remark}
\newtheorem{corollary}[theorem]{Corollary}

\newcommand{\esssup}{\mathop{\rm ess\,sup}}

\usepackage{cite}

\ifCLASSINFOpdf
\usepackage[pdftex]{graphicx}
\else
\fi
\usepackage[cmex10]{amsmath}
\usepackage{array}
\usepackage{url}

\begin{document}
	%
	\title{An LMI Approach to Stability Analysis of Coupled Parabolic Systems \thanks{
}
	}
	%
	%
	%
	
	
	\author{Masashi~Wakaiki,~\IEEEmembership{Member,~IEEE}
		\thanks{
			M.~Wakaiki is with the 
			Graduate School of System Informatics, Kobe University, Hyogo 657-8501, Japan.
			(email:{\tt  
				wakaiki@ruby.kobe-u.ac.jp).}}
		\thanks{
			This work was supported by JSPS KAKENHI Grant Numbers
			JP17K14699.}
	}

	\maketitle
	
	\begin{abstract}
We analyze the exponential stability
of a class of distributed parameter systems.
The system we consider is described by a coupled
parabolic partial differential equation with spatially varying coefficients.
We approximate the coefficients
by splitting space domains but
take into account approximation errors during stability analysis. 
Using a quadratic Lyapunov function, we obtain sufficient conditions for
exponential stability in terms of linear matrix inequalities.
	\end{abstract}
	
	
	%
	\IEEEpeerreviewmaketitle

\section{Introduction}
Consider the following parabolic partial differential equation (PDE) on 
a bounded open set
$\Omega \subset \mathbb{R}^m$:
\begin{equation}
\label{eq:coupled_PDE}
\begin{cases}
\partial_t z = A\Delta z+ B(x)z
& \text{in $\Omega \times (0,T]$} \\
z = 0 & \text{on $\partial \Omega \times (0,T]$} \\
z(\cdot,0) = z^0 & \text{in $\Omega$},
\end{cases}
\end{equation}
where $A \in \mathbb{R}^{n \times n}$,
$B \in L^{\infty}(\Omega)^{n \times n}$, 
$\Delta$ is the Laplacian acting componentwise,
$z = [z_1,\dots,z_n]^{\top}: 
\overline \Omega \times [0,T] \to \mathbb{R}^n$ is the state, 
and
$z^0: \Omega \to \mathbb{R}^n$ is a given initial data in $L^2(\Omega)^n$. 
This PDE is a subclass of abstract parabolic equations (see, e.g.,  Sect.~11.1 of \cite{Renardy1993}),
and we call the PDE in \eqref{eq:coupled_PDE} a {\em coupled parabolic system}.
In this paper,
we study the stability analysis of this coupled parabolic system \eqref{eq:coupled_PDE}
by using linear matrix inequalities (LMIs) and a quadratic Lyapunov function
\begin{equation}
\label{eq:Lyap_intro}
V(z) := \int_{\Omega} z(x)^{\top} P(x) z(x) dx \qquad \forall z \in L^2(\Omega)^n.
\end{equation}

Lyapunov-based stability analysis without approximation has
recently been developed for distributed parameter systems.
The authors 
of \cite{Valmorbida2015CDC, Valmorbida2016, Gahlawat2017, Gahlawat2017CDC} have proposed semi-definite programming
approaches for the stability of 1-D and 2-D PDEs with polynomial data. 
LMI-based exponential stability conditions have been obtained for
various classes of distributed parameter systems, for example, 
1-D heat/wave equations with time-varying delays in \cite{Fridman2009PDE},
1-D semilinear parabolic systems in \cite{Fridman2012},
coupled n-D semilinear diffusion equations with time-varying delays
(and spatially constant coefficients) in \cite{Solomon2015}, and
n-D wave equations in \cite{Fridman2016}. Moreover,
in \cite{Diagne2012}, a sufficient dissipative boundary condition
has been derived to guarantee the 
exponential stability of coupled 1-D hyperbolic systems.
In terms of the coupled parabolic system \eqref{eq:coupled_PDE},
its controllability
have been extensively investigated, e.g., in \cite{Khodja2011, Cara2015}.
However,
relatively little work has been done on the stability analysis of this 
class of distributed parameter systems.
The difficulties in the stability analysis of the parabolic system \eqref{eq:coupled_PDE}
are the following three points: 
(i) The state $z$ is vector-valued.
(ii) The coefficient function $B$ may not be constant or even polynomial.
(iii) The set $\Omega$ is multi-dimensional.

To the parabolic PDE \eqref{eq:coupled_PDE},
we apply the gridding methods that have been
proposed for establishing the stability of networked control systems with
aperiodic sampling and time-varying delays, e.g.,  in \cite{Fujioka2009,Oishi2010, Donkers2011,Donkers2012_stochastic, Hetel2011}.
First we consider a general bounded  open set $\Omega$ and
transform the coefficient function $B(x)$ to a piecewise constant function plus
an approximation error by splitting the set $\Omega$.
This approximation error is taken
into account during the stability analysis.
We obtain an LMI-based sufficient condition for exponential stability, 
using a Lyapunov function in \eqref{eq:Lyap_intro} where
$P(x)$ is a constant function.
Second, we focus on the case where the set $\Omega$ is a polytope.
In this case, we approximate $B(x)$ by a piecewise linear function and
use a Lyapunov function in \eqref{eq:Lyap_intro} where
$P(x)$ is piecewise linear on $\Omega$.
This means that we use a wider class of Lyapunov functions
to analyze the stability of the coupled parabolic 
system \eqref{eq:coupled_PDE}.
As a result,
we can obtain a less conservative sufficient LMI condition
for exponential stability in the case of a polytope $\Omega$.


This paper is organized as follows.
In Section~II, we recall preliminary results on Sobolev spaces and
the concept of weak solutions of
the parabolic PDE \eqref{eq:coupled_PDE}.  In Section~III,
we analyze the exponential stability of the coupled PDE \eqref{eq:coupled_PDE}
with general set $\Omega$, by using Lyapunov functions with constant $P$.
In Section IV,  stability analysis by Lyapunov functions 
with piecewise linear $P$
is presented
in the case where $\Omega$ is a polytope.
We illustrate numerical examples in Section~V.

\paragraph*{Notation}
For a set $\Omega \subset \mathbb{R}^m$, its closure, interior, and boundary are denoted by 
$\overline{\Omega}$, $\Omega^i$, and $\partial \Omega$, respectively.
Let us denote the Euclidean norm of a vector $\xi \in \mathbb{R}^m$ by $\|\xi\|$.
For a matrix $M \in \mathbb{R}^{n \times p}$, 
we denote by $M^{\top}$ and $\|M\|$ its transpose and Euclidean-induced norm, respectively.
Let us denote by $\{e_\ell\}_{\ell=1}^m$ the standard basis in 
$\mathbb{R}^m$, namely, 
$e_1 = \begin{bmatrix}
1 &  0 & \cdots & 0
\end{bmatrix}^{\top}$, \dots,
$e_m = \begin{bmatrix}
0 & \cdots & 0 & 1
\end{bmatrix}^{\top}$.
For a square matrix $P \in \mathbb{R}^{n \times n}$,
the notation $P \succ 0$ means that $P $ is symmetric and
positive definite.
The Kronecker
product of two real matrices $A$ and $B$ is denoted by $A \otimes B$.
For simplicity, 
we write a partitioned real symmetric matrix 
\[
\begin{bmatrix}
A & B \\ B^\top & C
\end{bmatrix}
\text{~~as~~}
\begin{bmatrix}
A & B \\ \star & C
\end{bmatrix}.
\]

Let $\Omega \subset \mathbb{R}^m$ be an open set.
We denote by
$L^2(\Omega)^n$
the space of all measureable functions $f:\Omega \to \mathbb{R}^n$ satisfying $\int_\Omega \|f(x)\|^2 dx < \infty$.
The norm and inner product of $L^2(\Omega)^n$ are  defined by
\[
\|f\|_{L^2(\Omega)^n} := \sqrt{\int_\Omega \|f(x)\|^2 dx},~~
(f,g)_{L^2(\Omega)} := \int_\Omega f(x)^{\top}g(x) dx,
\]
respectively.
The space
$L^{\infty}(\Omega)^{n \times p}$ 
consists of all measureable functions $F:\Omega \to \mathbb{R}^{n \times p}$ satisfying 
$\esssup_{x \in \Omega} \|F(x)\| < \infty$.
The norm of $L^{\infty}(\Omega)^{n \times p}$ is defined by
\[\|F\|_{L^\infty(\Omega)^{n \times p}} := \esssup_{x \in \Omega} \|F(x)\|.
\]
We write $L^{\infty}(\Omega)^{n}$ for $L^{\infty}(\Omega)^{n \times 1}$, and
if $n=1$, then we will drop the superscript $n$.
Let us denote by $H^1(\Omega)^n$ the space of all functions 
$f = 
\begin{bmatrix}
	f_1 &  f_2 & \cdots & f_n
\end{bmatrix}^{\top} \in
L^2(\Omega)^n$ 
such that 
the first-order partial derivatives of $f_1,\dots,f_n$ exist in the weak sense 
and belong to $L^2(\Omega)$. We denote the gradient of 
a scalar-valued function
$f \in H^1(\Omega)$ by $\nabla f$ and define in the vector-valued case,
\[
\nabla f
:= 
\begin{bmatrix}
\nabla f_1\\
\vdots \\
\nabla f_n
\end{bmatrix}
\qquad
\text{for~}f = \begin{bmatrix} f_1 \\ \vdots \\ f_n \end{bmatrix} \in H^1(\Omega)^n.
\]
The norm of  $H^1(\Omega)^n$ is defined by 
\[\|f\|_{H^1(\Omega)^n} := \sqrt{\|f\|_{L^2(\Omega)^n}^2 + \|\nabla f\|_{L^2(\Omega)^{mn}}^2}.
\]
The space $C^\infty_0(\Omega)^n$ comprises all infinitely many times differentiable functions $f:\Omega \to \mathbb{R}^n$
such that $\text{supp}~\!f := \overline{\{x \in \Omega: f(x) \not=0\}}$ is compactly contained in $\Omega$.
The space $H^1_0(\Omega)^n$ is the closure of $C^{\infty}_0(\Omega)^n$ in $H^1(\Omega)^n$.
We denote by $H^{-1}(\Omega)^n$ the dual space of $H^1_0(\Omega)^n$, that is,
the space of bounded linear maps $g:H^1_0(\Omega)^n \to \mathbb{R}$.
Elements of $H^{-1}(\Omega)^n$ can be regarded as $n$-dimensional vectors whose 
entries belong to $H^{-1}(\Omega)$.
The duality pairing between $H^{1}_0(\Omega)^n$ and its dual
$ H^{-1}(\Omega)^n$ is denoted by
$
\langle
g,f
\rangle : H^{-1}(\Omega)^n \times H^1_0(\Omega)^n \to \mathbb{R}.
$
For simplicity of notation,
we will 
drop the dimension $n$ and the set $\Omega$ in the norms and the inner product, e.g., write $\|f\|_{H^1}$
for $\|f\|_{H^1(\Omega)^n}$. 

Let $X$ be a Banach space with norm $\|\cdot\|_X$. We denote by $L^2(0,T;X)$ the space of all
(strongly) measurable functions $f:(0,T) \to X$ such that $\int_0^T \|f\|_X^2 dt < \infty$. 
The space $C\big([0,T];X\big)$ comprises all continuous functions $f:[0,T] \to X$.

\section{Preliminaries}


In what follows, we write $z(t) = z(\cdot, t)$ from the vector-valued viewpoint.
The following theorem will be useful to study the
solution of the coupled parabolic PDE \eqref{eq:coupled_PDE}:
\begin{theorem}[Sec.~5.9.2 in \cite{Evans1998}]
	\label{thm:z_continuity}
	Let
	\[z \in L^2\big(0,T;H^1_0(\Omega)^n\big)
	\text{~and~~}
	\frac{dz}{dt}\in L^2\big(0,T;H^{-1}(\Omega)^n\big).
	\] 
	Then \\ \noindent
	(i) $z \in C\big([0,T];L^2(\Omega)^n\big)$;
	\\ \noindent
	(ii) The mapping $t \mapsto \|z(t)\|_{L^2}^2$ is absolutely continuous with
	\[
	\frac{d}{dt} \|z(t)\|^2_{L^2} = 2 
	\left \langle 
	\frac{dz}{dt}(t) , z(t)
	\right\rangle\qquad \text{a.e.~} t \in [0,T].
	\] 
\end{theorem}

Although only the case $n=1$ is considered in Sec.~5.9.2 in \cite{Evans1998},
one can obtain Theorem~\ref{thm:z_continuity}, the case $n\geq 1$, by applying the result of the case $n=1$ to
 each element of $z$.

We define a weak solution of the coupled parabolic PDE \eqref{eq:coupled_PDE}.
\begin{definition}[Weak solution]
	\label{defn:weak_solution}
	A function $z \in L^2\big(0,T; H^1_0(\Omega)^{n}\big)$ with
	$\frac{dz}{dt} \in L^2\big(0,T; H^{-1}(\Omega)^{n}\big)$
	 is a weak solution 
	of the coupled parabolic PDE \eqref{eq:coupled_PDE} with 
	the initial data $z^0 \in L^2(\Omega)^n$ if 
	the following two conditions hold:
	\begin{enumerate}
		\item For every $v \in H^1_0 (\Omega)^n$ and 
		for a.e. $t \in [0,T]$,
		\begin{equation*}
		\left\langle
		\frac{dz}{dt}(t), v
		\right\rangle = 
		- \big((A \otimes I_m) \nabla z(t), \nabla v\big)_{L^2} + \big(Bz(t),v\big)_{L^2};
		\end{equation*}
		\item $z(0) = z^0$.
	\end{enumerate}
%
\end{definition}

Since $z \in L^2(0,T; H^1_0(\Omega)^{n})$ and 
$\frac{dz}{dt} \in L^2(0,T; H^{-1}(\Omega)^{n})$ in Definition~\ref{defn:weak_solution},
it follows that
(i) of Theorem~\ref{thm:z_continuity} yields $z \in C\big([0,T]; L^2(\Omega)^{n}\big)$.
Hence the initial condition 2) makes sense.

We place the following coercivity condition on the coefficient matrix $A$ in 
the PDE \eqref{eq:coupled_PDE}, which is used to guarantee the existence and 
uniqueness of weak solutions.
\begin{assumption}
	\label{assum:A}
	There exists $\alpha >0$ such that 
	$A \in \mathbb{R}^{n \times n}$  in the PDE \eqref{eq:coupled_PDE} 
	satisfies $\zeta^{\top} A \zeta \geq \alpha \|\zeta\|^2$ for every $\zeta \in \mathbb{R}^n$
\end{assumption}

Applying Galerkin's method,
we see that  if Assumption~\ref{assum:A} is satisfied, then
for every initial data $z^0 \in L^2(\Omega)^n$, there exists a unique weak solution
of the coupled parabolic PDE \eqref{eq:coupled_PDE};
see, e.g., Sec.~7.1 of \cite{Evans1998} and
Sec.~11.1 of \cite{Renardy1993}.
To make the paper  self-contained, we provide the proof of the existence and uniqueness of 
weak solutions in the appendix.

We define the exponential stability of 
the coupled parabolic system \eqref{eq:coupled_PDE}.
\begin{definition}[Exponential stability]
	The coupled parabolic system \eqref{eq:coupled_PDE}
	is exponentially stable if 
	there exist $M \geq 1$ and $\gamma >0$ such that
	for each $T > 0$,
	the weak solution $z$ of the PDE \eqref{eq:coupled_PDE} 
	satisfies
	\[
	\|z(t)\|_{L^2} \leq M e^{-\gamma t} \|z^0\|_{L^2} ~~\quad
	\forall z^0 \in L^2(\Omega)^n,~
	\forall t \in [0,T].
	\]
\end{definition}


To analyze the exponential stability of the coupled PDE, we employ Poincar\'e-Friedrichs' inequality.
\begin{theorem}[Poincar\'e-Friedrichs' inequality]
	\label{thm:poincare_ineq}
	For every bounded open set $\Omega \subset \mathbb{R}^m$,
	there exists a constant $c= c(\Omega)> 0$ such that 
	\begin{equation}
	\label{eq:poincare_ineq}
	\|z\|_{L^2} \leq c \|\nabla z\|_{L^2}\qquad \forall z \in H^1_0(\Omega).
	\end{equation}
\end{theorem}

If $\Omega$ is contained between a
pair of parallel hyperplanes situated at a distance $\delta >0$, then
the constant $c$ of Poincar\'e-Friedrichs' inequality is given by
$\delta $; see, e.g.,  Proposition 13.4.10  in \cite{Tucsnak2009}.
If $\Omega = (a,b) \subset \mathbb{R}$, then
$c = (b-a) / \pi$, which
cannot be improved; see, e.g.,  Sec.~1.7.2  in \cite{Dym1972}.

Applying 
Poincar\'e-Friedrichs' inequality
to each element of $z \in H^1_0(\Omega)^{n}$,
we obtain the following result:
\begin{corollary}
	\label{coro:poincare}
	Let $\Omega \subset \mathbb{R}^m$ be bounded and open.
	For every $z \in H^1_0(\Omega)^{n}$ and 
	every positive definite diagonal matrix $\Lambda \in \mathbb{R}^{n \times n}$,
	we obtain 
	\begin{equation}
	\label{eq:S_pro}
	\int_{\Omega} 
	\begin{bmatrix}
	z(x) \\ \nabla z(x)
	\end{bmatrix}^{\top}
	\begin{bmatrix}
	-\Lambda & 0 \\
	0 & c^2 \Lambda \otimes I_m
	\end{bmatrix}
	\begin{bmatrix}
	z(x) \\ \nabla z(x)
	\end{bmatrix}dx \geq 0,
	\end{equation}	
	where $c >0$ is a constant of the 
	Poincar\'e-Friedrichs' inequality \eqref{eq:poincare_ineq}.
\end{corollary}

\section{Stability Analysis by Lyapunov Functions with Constant $P$}
First we study the stability of the coupled parabolic PDE \eqref{eq:coupled_PDE},
by using Lyapunov function with constant $P$.
We place the following assumptions on  the bounded open set $\Omega$:
\begin{assumption}
	\label{assum:general_set}
	For a bounded open set $\Omega \subset \mathbb{R}^m$, let 
	Lebesgue measurable
	sets $\Omega_k \subset \mathbb{R}^m$ ($k=1,\dots,N$) 
	satisfy the following conditions:
	\begin{flalign*} & \text{(i)}~
	\Omega = \bigcup_{k=1}^N \Omega_k;& \\
	& \text{(ii)}~\Omega_{k} \cap \Omega_{\ell} = \emptyset \qquad \forall k,\ell=
	1,\dots,N \text{~with~$k\not=\ell$}.&
	\end{flalign*}
\end{assumption}
\begin{assumption}
	\label{assum:general_set2}
	For every $k=1,\dots,N$, 
	the matrix $B_k \in \mathbb{R}^{n \times n}$ and the scalar $\rho_k > 0$ satisfy
	\begin{equation}
	\label{eq:B_bound}
	\|B(x) - B_k\| \leq \rho_k
	\qquad \text{a.e.~} x \in \Omega_k.
	\end{equation}	
\end{assumption}

For example, we can choose $B_k = B(\omega_k)$, where
$\omega_k \in \mathbb{R}^m$ is the ``center'' of $\Omega_k$.
The scalar $\rho_k$ is the approximation error of $B_k$.
The disjoint subsets $\Omega_1,\dots,\Omega_N$ are
tuning parameters in our stability analysis. 

We then have the following sufficient LMI condition for stability:
\begin{theorem}
	\label{thm:stability}
	Let Assumptions~\ref{assum:A}, \ref{assum:general_set}, and \ref{assum:general_set2} hold.
	The coupled parabolic system \eqref{eq:coupled_PDE}
	is exponentially stable if
	there exist a positive definite matrix $P \in \mathbb{R}^{n \times n}$,
	a positive definite diagonal matrix $\Lambda\in \mathbb{R}^{n \times n}$,
	and a positive scalar $\sigma_k$ such that
	the following LMIs are feasible for all $k = 1,\dots,N$:
	\begin{subequations}
		\label{eq:stability_LMI}
		\begin{align}
		\label{eq:LMI1}
		\begin{bmatrix}
		\Lambda - \sigma_k I_n- B_k^{\top}P-PB_k \quad & \rho_k P \\
		\star \quad& \sigma_k I_n
		\end{bmatrix} &\succ 0 \\
		\label{eq:LMI2}
		A^{\top} P + PA^{\top} - 
		c^2 \Lambda &\succeq 0,
		\end{align}
		where $c >0$ is a constant of the 
		Poincar\'e-Friedrichs' inequality \eqref{eq:poincare_ineq}.
	\end{subequations}
\end{theorem}
\begin{proof}
	1. 
	Using the positive definite matrix $P \in \mathbb{R}^{n \times n}$, we define 
	\begin{equation*}
	V(z) := (z,Pz)_{L^2} \qquad \forall z \in L^2(\Omega)^n.
	\end{equation*}
	%
	We use the notation $V(t) := V\big(z(t)\big)$ for simplicity, where
	$z$ is the weak solution of the PDE \eqref{eq:coupled_PDE}.
	One can see that 
	$V(t)$ is  absolutely continuous on $[0,T]$ and 
	\[
	\frac{dV}{dt} (t)= 2 
	\left \langle 
	\frac{dz}{dt}(t) , Pz(t)
	\right\rangle\qquad \text{a.e.~} t \in [0,T]
	\]
	from the same argument as in the proof of
	(ii) of Theorem~\ref{thm:z_continuity}.
	Since $Pz(t) \in H^1_0(\Omega)^n$ for a.e. $t \in [0,T]$,
	it follows from the condition~1) in Definition~\ref{defn:weak_solution} that
	\begin{align}
	\label{eq:V_derivative}
	\frac{dV}{dt}(t)
	=
	\int_{\Omega} 
	\begin{bmatrix}
	z(x,t) \\ \nabla z(x,t)
	\end{bmatrix}^{\top} 
	M(x)
	\begin{bmatrix}
	z(x,t) \\ \nabla z(x,t)
	\end{bmatrix}dx
	\end{align}
	for a.e. $t \in [0,T]$, 
	where 
	\begin{align*}
	M(x) :=  
	\begin{bmatrix}
	B(x)^{\top} P + PB(x) & 0 \\
	0 &  -(A^{\top}P + P A)  \otimes I_m
	\end{bmatrix}.
	\end{align*}
	On the other hand, 
	since $z(\cdot,t) \in H^1_0(\Omega)^n$ for a.e. $t \in [0,T]$,
	Corollary~\ref{coro:poincare} shows 
	that for a.e.  $t \in [0,T]$, 
	\begin{equation}
	\label{eq:S_pro_constant}
	\int_{\Omega} 
	\begin{bmatrix}
	z(x,t) \\ \nabla z(x,t)
	\end{bmatrix}^{\top}
	\begin{bmatrix}
	-\Lambda & 0 \\
	0 & c^2 \Lambda \otimes I_m
	\end{bmatrix}
	\begin{bmatrix}
	z(x,t) \\ \nabla z(x,t)
	\end{bmatrix}dx \geq 0.
	\end{equation}
	Therefore, if there exists $\epsilon >0$ such that 
	for a.e. $x \in \Omega$,
	\begin{subequations}
		\label{eq:Lambda_x}
		\begin{align}
		\label{eq:BP_Lambda_x}
		\Lambda - B(x)^{\top} P - PB(x) &\succeq \epsilon I_n \\
		\label{eq:AP_Lambda_x}
		(A^{\top} P + PA - c^2 \Lambda) \otimes I_m&\succeq 0,
		\end{align}
	\end{subequations}
	then it follows from \eqref{eq:V_derivative} and \eqref{eq:S_pro_constant} that
	\begin{equation}
	\label{eq:V_deriv_z}
	\frac{dV}{dt}(t) \leq
	-\epsilon \|z(t)\|_{L^2}^2 \qquad \text{a.e.~} t \in [0,T].
	\end{equation}
	
	2. 
	We next show that if the LMIs
	\eqref{eq:stability_LMI} are feasible, then
	there exists a constant $\epsilon >0$ such that
	the inequalities \eqref{eq:Lambda_x} 
	hold for a.e. $x \in \Omega$.                                  
	By the LMI \eqref{eq:LMI1},
	there exists $\epsilon >0$
	such that for all $k=1,\dots,N$,
	\begin{align}
	\label{eq:LMI1_eps}
	G_k := \begin{bmatrix}
	\Lambda - (\sigma_k+\epsilon) I_n- B_k^{\top}P-PB_k  \quad & \rho_k P \\
	\star  \quad& \sigma_k I_n
	\end{bmatrix} \succ 0.
	\end{align}
	By \eqref{eq:B_bound}, for every $k=1,\dots,N$,
	there exists a measurable function $\Phi_k:\Omega_k \to \mathbb{R}^{n \times n}$ such that
	for a.e. $x \in \Omega_k$,
	\begin{subequations}
		\label{eq:Bx_Phi_cond}
		\begin{align}
		\label{eq:Bx_Phi}
		B(x)- B_k &= \rho_k \Phi_k(x) \\
		\label{eq:Phi_bound}
		\|\Phi_k(x)\| &\leq 1.
		\end{align}
	\end{subequations}
	Since $I_n - \Phi_k^{\top}(x) \Phi_k(x) \succeq 0$ for a.e. $x \in \Omega_k$ 
	by \eqref{eq:Phi_bound}, 
	it follows from \eqref{eq:Bx_Phi} that
	\begin{align*}
	&\begin{bmatrix}
	I_n \\ -\Phi_k(x)
	\end{bmatrix}^{\top}
	G_k
	\begin{bmatrix}
	I_n \\ -\Phi_k(x)
	\end{bmatrix}  \\
	&\quad =
	\Lambda - \big(B_k+\rho_k\Phi_k(x)\big)^{\top}P  -  P\big(B_k+\rho_k\Phi_k(x)\big) - \epsilon I_n  \\
	&\qquad \qquad - \sigma_k \big(I_n -\Phi_k(x)^{\top} \Phi_k(x)\big) \\
	&\quad \preceq 
	\Lambda - B(x)^{\top}P  - PB(x)-\ \epsilon I_n 
	~~~ \text{a.e.~} x \in \Omega_k. 
	\end{align*} 
	Thus,
	the inequality 
	\eqref{eq:LMI1_eps}
	yields \eqref{eq:BP_Lambda_x}.
	Moreover, 
	since $(A^{\top} P + PA - c^2\Lambda) \otimes I_m$ and $A^{\top} P + PA - c^2\Lambda$ have the same eigenvalues,
	it follows that \eqref{eq:LMI2}  implies
	\eqref{eq:AP_Lambda_x}.
	
	3. 
	Finally, we show that 
	the inequality  \eqref{eq:V_deriv_z} leads to
	the exponential stability of the coupled parabolic
	system \eqref{eq:coupled_PDE}.
	Let $\delta_{\min}$ and $\delta_{\max}$ be 
	the minimum and maximum eigenvalues of $P$, respectively.
	From the inequality \eqref{eq:V_deriv_z}, we find that
	\begin{align}
	\label{eq:dV_bound}
	\frac{dV}{dt}(t) 
	&\leq -2\gamma V(t)
	\quad \text{a.e.~} t \in [0,T],
	\end{align}
	where 
	$\gamma := \epsilon/(2\delta_{\max})$.
	Since $V(t)$ is absolutely continuous on $[0,T]$,
	Gronwall's inequality (see, e.g., 
	Appendix B.2.j  in \cite{Evans1998}) yields
	\[
	V(t) \leq V(0)e^{-2\gamma t}\qquad \forall t \in [0,T].
	\]
	Thus, for each $T > 0$ and
	each initial state $z^0 \in L^2(\Omega)^n$,
	the solution $z(t)$ of the parabolic PDE \eqref{eq:coupled_PDE} satisfies
	\[
	\|z(t)\|_{L^2} \leq \sqrt{\frac{\delta_{\max}}{\delta_{\min}}}
	e^{-\gamma t} \|z^0\|_{L^2}\qquad \forall t \in [0,T].
	\]
	This completes the proof.
\end{proof}

\begin{remark}[Complexity of  LMIs in Theorem~\ref{thm:stability}]
	Let us study the numbers of variables in the LMIs of Theorem
	\ref{thm:stability}.
	In these LMIs, 
	the matrices $P$ and $\Lambda$ have $O(n^2)$ and $O(n)$ variables, respectively.
	On the other hand, the number of the scalar variables $\sigma_1,\dots,\sigma_N$ is $O(N)$.
	In total,  the LMIs of Theorem
	\ref{thm:stability} contain $O(n^2+N)$ variables.
	Suppose that the number $N$ of the disjoint subsets $\Omega_1,\dots,\Omega_N$
	is given by $N=2^m$, which makes sense due to the curse of 
	dimensionality. Then the worst-case number is given by $O(n^2+2^m)$.
\end{remark}

\section{Stability Analysis by Lyapunov Functions with Piecewise Linear $P$}
In this section, we analyze the stability of the coupled parabolic system \eqref{eq:coupled_PDE}, 
by using Lyapunov functions that depend on the space variable in a
piecewise linear fashion.
We impose the following assumption on the bounded open set $\Omega$:
\begin{assumption}
	\label{assum:polytopic_set}
	For a bounded open set $\Omega \subset \mathbb{R}^m$, let
	$m$-simplices $\Omega_1,\dots,\Omega_N \subset \mathbb{R}^m$ satisfy
	the following conditions:
	\begin{flalign*} & \text{(i)}~
	\overline{\Omega} = \bigcup_{k=1}^N \Omega_k ;& \\
	& \text{(ii)}~\Omega_j \cap \Omega_k
	\not=\emptyset 
	~~\Rightarrow~~
	\text{$\Omega_j \cap \Omega_k$ is a face of $\Omega_j$
		and $\Omega_k$.}&
	\end{flalign*}
\end{assumption}

For $k = 1,\dots,N$, 
let $\xi_0^k,\dots,\xi_m^k$ be the vertices of the $m$-simplex $\Omega_k$.
We reorder the set $\{\xi_0^k,\dots,\xi_m^k:k =1\dots,N\}$  into
$\{\xi_1\dots,\xi_{N_0}\}$ without duplication. Namely, 
$\{\xi_0^k,\dots,\xi_m^k:k =1\dots,N\} = \{\xi_1\dots,\xi_{N_0}\}
$
and $\xi_k \not= \xi_\ell$ for every $k,\ell = 1,\dots,N_0$ with $k\not=\ell$.
Let $\xi_{p(k,0)},\dots, \xi_{p(k,m)}$ be the vertices of $\Omega_k$
for every $k = 1,\dots,N$. Define a matrix $B_p := B(\xi_p)$ for every $p=1,\dots,N_0$, and
let $x \in \Omega_k$ be represented as  
\begin{equation}
\label{eq:alpha_simplex}
x = \sum_{\ell=0}^m \alpha_{p(k,\ell)}(x) \xi_{p(k,\ell)},
\end{equation}
where the coefficients
$\alpha_{p(k,0)}(x),\dots,\alpha_{p(k,m)}(x)$
are nonnegative and satisfy
$\sum_{\ell=0}^m \alpha_{p(k,\ell)}(x) = 1$.

We formulate the remaining assumption.
\begin{assumption}
	\label{assum:polytopic_set2}
	For every $k=1,\dots,N$,  a scalar $\rho_k >0$ satisfies
	\begin{equation}
	\label{eq:B_bound2}
	\left\| 
	B(x) - \sum_{\ell=0}^m \alpha_{p(k,\ell)} (x) B_{p(k,\ell)}  
	\right\| \leq \rho_k\qquad \text{a.e.~}x\in \Omega_k.
	\end{equation}
\end{assumption}

As in Assumption~\ref{assum:general_set},
the tuning of the disjoint sets $\Omega_1,\dots\Omega_N$ is needed
in our stability analysis to obtain a less conservative result.

For the second main result,
we use the following lemma on LMIs, inspired by 
the stability analysis of systems with polytopic uncertainty developed, e.g., 
in \cite{Peaucelle2000}:
\begin{lemma}
	\label{lem:LMI_formula}
	For every  symmetric matrix $M$ and
	for every matrices $B$ and $P$,
	the inequality
	\begin{equation}
	\label{eq:up_xi_expansion1}
	M - B^{\top}P - P^{\top}B \succeq 0
	\end{equation}
	is satisfied
	if and only if there exist 
	(not necessarily symmetric) 
	matrices $\Upsilon$ and $\Xi$ such that
	\begin{equation}
	\label{eq:up_xi_expansion2}
	\begin{bmatrix}
	M - B^{\top}\Upsilon - \Upsilon^{\top} B \quad 
	& \Upsilon^{\top} - P^{\top} - B^{\top}\Xi \\
	\star \quad & \Xi + \Xi^{\top}
	\end{bmatrix} \succeq 0.
	\end{equation}
\end{lemma}
\begin{proof}
	Since
	\begin{align*}
	&\begin{bmatrix}
	I \\ B
	\end{bmatrix}^{\top} 
	\begin{bmatrix}
	M - B^{\top}\Upsilon - \Upsilon^{\top} B \quad & \Upsilon^{\top} - P^{\top} - B^{\top}\Xi \\
	\star \quad& \Xi + \Xi^{\top}
	\end{bmatrix}
	\begin{bmatrix}
	I \\ B
	\end{bmatrix} \\
	&\quad= M - B^{\top}P - P^{\top}B ,
	\end{align*}
	it follows that \eqref{eq:up_xi_expansion2} leads to
	\eqref{eq:up_xi_expansion1}.
	On the other hand,
	\eqref{eq:up_xi_expansion2} with $\Upsilon = P$ and $\Xi = 0$
	is equivalent to \eqref{eq:up_xi_expansion1}.
	This completes the proof.
\end{proof}

The following lemma provides a sufficient condition 
for the product of Sobolev functions to belong to $H^1_0$.
\begin{lemma}
	\label{lem:H10}
	For a bounded open set $\Omega \subset \mathbb{R}^m$,
	if $f \in H^1(\Omega) \cap L^{\infty}(\Omega)$ with
	$\nabla f \in L^{\infty}(\Omega)^m$ and if $g \in H^1_0(\Omega)$,
	then $f g \in H^1_0(\Omega)$.
\end{lemma}
\begin{proof}
	First we show $fg \in  H^1(\Omega)$. Since 
	$f \in H^1(\Omega) $ and $g \in H^1_0(\Omega)$,
	it follows that $fg$ possesses weak derivatives and $\nabla (fg) = 
	(\nabla f)g + f (\nabla g)$. Recall that for every $v,w \in \mathbb{R}^m$,
	\begin{equation}
	\label{eq:ele_inequality}
	\|v+w\|^2 \leq \big(\|v\|+\|w\|\big)^2\leq  2\big( \|v\|^2 + \|w\|^2 \big).
	\end{equation}
	Since
	$f \in L^{\infty}(\Omega)$ and $\nabla f \in L^{\infty}(\Omega)^m$
	yield
	\begin{align*}
	\int_\Omega |f(x)g(x)|^2 dx &\leq \|f\|_{L^{\infty}(\Omega)}^2 \cdot \|g\|^2_{H^1(\Omega)} \\
	 	\frac{1}{2}\!\int_\Omega \! \|\nabla(fg)(x)\|^2 dx &\leq \! 
	\int_{\Omega}\! \|\nabla f(x) g(x)\|^2
	 dx \!+\!\! \int_{\Omega} \!\|f(x) \nabla g(x)\|^2 dx \\
	 &\leq 
	 \Big(
	 \|f\|_{L^{\infty}(\Omega)}^2 + \|\nabla f\|_{L^{\infty}(\Omega)}^2
	 \Big) \|g\|^2_{H^1(\Omega)},
 	\end{align*}
 	it follows that $fg \in H^1(\Omega)$.
 	
 	To show $fg \in H^1_0(\Omega) = \overline{C^\infty_0(\Omega)}$, it is enough to prove that 
 	for every $\epsilon >0$, there exists $h \in C^{\infty}_0(\Omega)$ such that 
 	\begin{equation}
 	\label{eq:h_fg_diff}
 	\|h - fg\|_{H^1(\Omega)} < \epsilon.
 	\end{equation}
 	Fix $\epsilon >0$ arbitrarily. Since $g \in H^1_0(\Omega)$, 
	there exists $g_{\epsilon} \in C^{\infty}_0(\Omega)$ such that 
	\[
	\|g_\epsilon - g\|^2_{H^1(\Omega)} \leq \frac{\epsilon^2}{12\big(
		\|f\|_{L^{\infty}(\Omega)}^2 + \|\nabla f\|_{L^{\infty}(\Omega)}^2
		\big)  }.
	\]
	Choose an open set $U$ such that $\text{supp}~\!g_\epsilon \subset U \subset \overline{U} \subset \Omega$.
	Using a mollifier (see, e.g., Sec. 1.1.5 of \cite{Mazja1985}), we obtain
	$f_\epsilon \in C^{\infty}(U)$ satisfying
	\[
		\|f_\epsilon- f\|^2_{H^1(U)} \leq \frac{\epsilon^2}{12\big(
		\|g_{\epsilon}\|_{L^{\infty}(\Omega)}^2 + \|\nabla g_{\epsilon}\|_{L^{\infty}(\Omega)}^2
		\big)  }.
	\]
	We define 
	\[
	h(x) := 
	\begin{cases}
	f_{\epsilon}(x) g_{\epsilon}(x) & \text{if~} x \in U \\
	0 & \text{if~} x \in \Omega \setminus U.
	\end{cases}
	\]
	Since $\text{supp}~\!g_\epsilon \subset U$, it follows that $h \in C^{\infty}_0(\Omega)$.
	
	Finally, we are in a position to prove \eqref{eq:h_fg_diff}. 
	Using \eqref{eq:ele_inequality} again, we find that 
	\begin{align*}
	&\frac{1}{2}
	\int_\Omega |h(x) - f(x)g(x)|^2 dx  \\
	&\leq
	\int_U\big|\big(f_\epsilon(x) - f(x)\big)g_\epsilon(x)\big|^2 dx 
	+ \int_\Omega\big|f(x)\big(g_\epsilon(x) - g(x)\big)\big|^2 dx \\
	&\leq 
	\|f_\epsilon- f\|^2_{H^1(U)} \cdot \|g_{\epsilon}\|_{L^\infty(\Omega)}^2
	+
	\|f\|_{L^\infty(\Omega)}^2 \cdot \|g_\epsilon- g\|^2_{H^1(\Omega)} \\
	&< \epsilon^2/6.
	\end{align*}
	Similarly,
	we obtain
	\begin{align*}
	&
	\frac{1}{2}\int_\Omega
	\big\|
	\nabla h(x) - \nabla(fg_\epsilon)(x)
	\big\|^2dx \\
	&\leq 
	\int_U\!
	\big\|
	\nabla 
	(f_\epsilon \!-\! f)(x)
	g_\epsilon (x)
	\big\|^2dx 	\!+\!
	\int_U\!
	\big\| 
	\big(f_\epsilon (x)\!-\! f(x) \big)
	\nabla g_\epsilon (x)
	\big\|^2dx	\\
	&\leq	
	\|f_\epsilon- f\|^2_{H^1(U)} \cdot \|g_{\epsilon}\|_{L^\infty(\Omega)}^2
	+ \|f_\epsilon- f\|^2_{H^1(U)} \cdot \|\nabla g_{\epsilon}\|_{L^\infty(\Omega)}^2 \\
	&< \epsilon^2 / 12
	\end{align*}
	and 
	\begin{align*}
	&\frac{1}{2}\int_\Omega
	\big\|
	\nabla(fg_\epsilon)(x) -
	\nabla(fg)(x)
	\big\|^2dx 	\\
	&\leq
	\int_\Omega \!
	\big\|
	\nabla f (x)
	\big(g_\epsilon(x) \!-\! g(x)\big)
	\big\|^2dx	\!+\!
	\int_\Omega \!
	\big\|
	f (x) \nabla
	(g_\epsilon \!-\! g)(x)
	\big\|^2dx 	\\
	&\leq 
	\|\nabla f\|_{L^\infty(\Omega)}^2 \cdot 
	\|g_\epsilon- g\|^2_{H^1(\Omega)} +
	\| f\|_{L^\infty(\Omega)}^2 \cdot 
	\|g_\epsilon- g\|^2_{H^1(\Omega)} \\
	&< \epsilon^2 / 12.
	\end{align*}
	Therefore, 
\begin{align*}
	&\frac{1}{2}
\int_\Omega
\big\|
\nabla 
h(x)-\nabla(fg)(x)
\big\|^2dx \\
&\leq \!
	\int_\Omega \!
\big\|
\nabla h(x) \!-\! \nabla(fg_\epsilon)(x)
\big\|^2dx  \!+\! 
	\int_\Omega \!
\big\|
\nabla(fg_\epsilon)(x) \!-\!
\nabla(fg)(x)
\big\|^2dx 	\\
&\leq \epsilon^2/3.
\end{align*}	
It follows that \eqref{eq:h_fg_diff} holds. Thus, we obtain $fg \in H^1_0(\Omega)$.
\end{proof}

The next result provides the partial derivatives of the coefficients of 
simplices.
\begin{lemma}
	\label{lem:alpha_grad}
	For an $m$-simplex $\Omega \subset \mathbb{R}^m$,
	let $\xi_0,\dots,\xi_m \in \mathbb{R}^m$ be its vertices and  
	$x \in \Omega$ be represented as
	\begin{equation}
	\label{eq:x_alpha_rep}
	x = \sum_{\ell=0}^m \alpha_{\ell}(x) \xi_{\ell},
	\end{equation}
	where the coefficients $\alpha_{0}(x),\dots,\alpha_m(x)$ are nonnegative and satisfy
	\begin{equation}
	\label{eq:alpha_property}
	\sum_{\ell=0}^m \alpha_{\ell}(x)= 1.
	\end{equation}
	Then, 
	the coefficients $\alpha_0(x),\cdots,\alpha_m(x)$ are continuous
	on $\Omega$. Furthermore,
	for every $x \in \Omega^i$,
	\begin{subequations}
		\label{eq:nabla_alpha_all}
		\begin{align}
		\label{eq:nabla_alpha_posi}
		\nabla \alpha_{\ell}(x) &= v_\ell\qquad \forall \ell = 1,\dots,m \\
		\label{eq:nabla_alpha_zero}
		\nabla \alpha_0(x) &= -\sum_{\ell=1}^{m} v_{\ell},
		\end{align}
		where 
		\begin{equation}
		\label{eq:v_def}
		v_{\ell} :=
		\begin{bmatrix}
		\xi_{1}^{\top} - \xi_{0}^{\top} \\
		\vdots \\
		\xi_{m}^{\top} - \xi_{0}^{\top}
		\end{bmatrix}^{-1} e_{\ell}\qquad \forall \ell = 1,\dots,m.
		\end{equation}
	\end{subequations}
\end{lemma}

\newcounter{mytempeqncnt2}
\begin{figure*}[t]
	\normalsize
	\setcounter{mytempeqncnt2}{\value{equation}}
	\setcounter{equation}{0}
	\begin{equation}
	\tag{$\sf A$}　\label{eq:LMI_1D}
	\begin{bmatrix}
	\Lambda  -  \sigma_{k,\ell} I_n  -  B_{p(k,\ell)}^{\top}\Upsilon_k  -  
	\Upsilon_k^{\top}B_{p(k,\ell)} \quad &
	\Upsilon_k^{\top}  -  P_{p(k,\ell)}  -  B_{p(k,\ell)}^{\top}\Xi_k \quad &  \rho_k P_{p(k,\ell)} & 
	\sum_{r = 0}^m (P_{p(k,r)}A)  \otimes   v_{p(k,r)}^{\top} \\
	\star \quad& \Xi_k+\Xi_k^{\top} \quad& 0 & 0 \\
	\star \quad& \star \quad& \sigma_{k,\ell} I_n & 0 \\
	\star \quad& \star \quad& \star & (A^{\top}P_{p(k,\ell)} + P_{p(k,\ell)}A  -  c^2\Lambda)  \otimes  I_m
	\end{bmatrix}  \succ  0
	\end{equation}
	\setcounter{equation}{\value{mytempeqncnt2}}
	\hrulefill
\end{figure*}

\begin{proof}
	By \eqref{eq:x_alpha_rep} and \eqref{eq:alpha_property},
	\[
	x = \left(1 - \sum_{\ell=1}^{m} 
	\alpha_\ell(x) \right)\xi_0 + \alpha_1(x)\xi_1 + \cdots + \alpha_m(x)\xi_m.
	\]
	Therefore,
	\begin{equation*}
	x - \xi_0 = 
	D
	\begin{bmatrix}
	\alpha_1(x) \\ \vdots \\ \alpha_m(x)
	\end{bmatrix},
	\end{equation*}
	where the matrix $D \in \mathbb{R}^{m \times m}$ is defined by
	\[
	D := 
	\begin{bmatrix}
	\xi_{1}- \xi_{0} & \cdots &
	\xi_{m} - \xi_{0} 
	\end{bmatrix}.
	\]
	Since $\Omega$ is an $m$-simplex, the matrix $D$ is invertible and
	\begin{equation}
	\label{eq:alpha_x_bounded}
	\begin{bmatrix}
	\alpha_1(x) \\ \vdots \\ \alpha_m(x)
	\end{bmatrix}	
	=
	D^{-1}
	(x-\xi_0).
	\end{equation}
	Hence
	$\alpha_1(x),\cdots,\alpha_m(x)$ are continuous on $\Omega$.
	By \eqref{eq:alpha_property}, 
	$\alpha_0(x) = 1 - \sum_{\ell=1}^m 
	\alpha_{\ell}(x)$ is also continuous on $\Omega$.
	
	Next, we investigate the gradients $\nabla \alpha_0,\dots,\nabla \alpha_m$.
	Choose $\ell=1,\dots,m$ arbitrarily.
	Since $\alpha_{\ell}(x) = e_{\ell}^{\top} D^{-1}(x-\xi_0)$ by \eqref{eq:alpha_x_bounded},
	it follows that 
	\begin{equation}
	\label{eq:nabla_alpha}
	\nabla \alpha_\ell(x) = 
	\begin{bmatrix}
	e_{\ell}^{\top}D^{-1} e_1 \\
	\vdots \\
	e_{\ell}^{\top}D^{-1} e_m
	\end{bmatrix}\qquad \forall x \in \Omega^i.
	\end{equation}
	The vector of the right-hand side of \eqref{eq:nabla_alpha}
	is equal to the transpose of the $\ell$th row vector of $D^{-1}$,
	namely, the vector $v_\ell$ defined by \eqref{eq:v_def}.
	Thus, we obtain \eqref{eq:nabla_alpha_posi}.
	Moreover, since
	\eqref{eq:alpha_property} leads to
	\[
	\nabla \alpha_0(x) = -\sum_{\ell=1}^m \nabla \alpha_{\ell}(x)\qquad \forall x \in \Omega^i,
	\]
	it follows that \eqref{eq:nabla_alpha_zero} holds. This completes the proof.
\end{proof}

For every $k=1,\dots,N$ and every $\ell=1,\dots,m$,
define
\[
v_{p(k,\ell)} :=
\begin{bmatrix}
\xi_{p(k,1)}^{\top} - \xi_{p(k,0)}^{\top} \\
\vdots \\
\xi_{p(k,m)}^{\top} - \xi_{p(k,0)}^{\top}
\end{bmatrix}^{-1} e_{\ell},~~
v_{p(k,0)} := -\sum_{\ell=1}^m v_{p(k,\ell)}.
\]
We are in a position to state the second main result.
\begin{theorem}
	\label{thm:piecewise_linear}
	Let Assumptions~\ref{assum:A}, \ref{assum:polytopic_set}, and \ref{assum:polytopic_set2} hold.
	The coupled parabolic system \eqref{eq:coupled_PDE}
	is exponentially stable if 
	there exist positive definite matrices $P_1,\dots,P_{N_0} \in \mathbb{R}^{n \times n}$,
	a positive definite diagonal matrix $\Lambda\in \mathbb{R}^{n \times n}$,
	and positive scalars $\sigma_{k,\ell}$ ($k=1,\dots,N$, $\ell=0,\dots,m$) such that
	the LMIs in \eqref{eq:LMI_1D}, where $c >0$ is a constant of the 
	Poincar\'e-Friedrichs' inequality \eqref{eq:poincare_ineq},
	are feasible for all $k = 1,\dots,N$ and $\ell=0,\dots,m$.
\end{theorem}

\begin{proof}	
	1.	
	Using $P_{p(k,\ell)}$, $B_{p(k,\ell)}$, and $\sigma_{k,\ell}$ in the LMIs in \eqref{eq:LMI_1D},
	we define the functions 
	${\bf P}_k$, ${\bf B}_k$, and ${\bm \sigma}_k$ on $\Omega_k$ by
	\begin{align*}
	{\bf P}_k(x) &:= \sum_{\ell=0}^m
	\alpha_{p(k,\ell)}(x)P_{p(k,\ell)} \\
	{\bf B}_k(x) &:= \sum_{\ell=0}^m
	\alpha_{p(k,\ell)}(x)B_{p(k,\ell)} \\
	{\bm \sigma}_k(x) &:=
	\sum_{\ell=0}^m
	\alpha_{p(k,\ell)}(x)\sigma_{k,\ell}
	\end{align*}
	for every $k=1,\dots,N$,
	where 
	the coefficients $\alpha_{p(k,0)},\dots, \alpha_{p(k,m)}$ are given as in \eqref{eq:alpha_simplex}.
	First we show that 
	the weak solution $z(t)$ of the parabolic PDE \eqref{eq:coupled_PDE}
	satisfies 
	$Pz(t) \in H^1_0(\Omega)^n$ for a.e. $t \in [0,T]$, where
	$P: \overline{\Omega} \to \mathbb{R}^{n \times n}$ is defined by
	\begin{equation}
	\label{eq:P_def}
	P(x) := {\bf P}_k(x) \qquad 
	\forall x \in \Omega_k,~\forall k = 1,\dots,N.
	\end{equation}
	
	To this end, we need to see that the values of ${\bf P}_j$ and ${\bf P}_k$ with 
	$\Omega_j \cap \Omega_k \not= \emptyset$ are not different
	on the intersection $\Omega_j \cap \Omega_k$.
	For every $j,k=1,\dots,N$ with 
	$\Omega_j \cap \Omega_k \not= \emptyset$,
	let 
	\[\xi_{q(0)},\dots,\xi_{q(m_0)} \in \{\xi_{p(j,\ell)}\}_{\ell=0}^{m} \cap \{\xi_{p(k,\ell)}\}_{\ell=0}^{m}
	\]
	be
	the vertices of the face $\Omega_j \cap \Omega_k$,
	which is guaranteed by (ii) of Assumption~\ref{assum:polytopic_set}.
	Then 
	for every $x \in \Omega_j \cap \Omega_k$,
	there exist
	\[
	\alpha_{q(0)}(x),\dots,\alpha_{q(m_0)}(x) \geq 0
	\text{~with~}
	\sum_{\ell=0}^{m_0} \alpha_{q(\ell)}(x) = 1,
	\] 
	such that 
	\[
	\sum_{\ell=0}^m
	\alpha_{p(j,\ell)}(x)P_{p(j,\ell)} 	=
		\sum_{\ell=0}^{m_0}
	\alpha_{q(\ell)}(x)P_{q(\ell)}=
	\sum_{\ell=0}^m
	\alpha_{p(k,\ell)}(x)P_{p(k,\ell)} .
	\]
	Hence, the values of $P$ are consistent on the boundaries.
	By Lemma~\ref{lem:alpha_grad},
	the coefficients $\alpha_{p(k,0)}(x),\dots,\alpha_{p(k,m)}(x)$ are 
	continuous in $\Omega_k$, which implies that 
	${\bf P}_k$ is continuous in $\Omega_k$ for every $k =1,\dots,N$.
	Thus $P$ is continuous in $\overline{\Omega}$. Furthermore,
	Lemma~\ref{lem:alpha_grad} shows that
	$\nabla \alpha_{p(k,\ell)}$ is constant for every $k=1,\dots,N$ and
	every $\ell = 0,\dots,m$.
	The restriction of each element of $P$ to every line parallel 
	to the coordinate directions is continuous piecewise linear and hence
	absolutely continuous. 
	Thus,
	every element of $P$ belongs to $H^1(\Omega)$
	by Theorem 2 in Sec.~1.1.3 of \cite{Mazja1985}, which is called the
	absolutely continuous on lines (ACL) characterization of Sobolev functions.
	Since $z(t) \in H^1_0(\Omega)^n$ for a.e. $t \in [0,T]$,
	it follows from Lemma~\ref{lem:H10} that
	$Pz(t) \in H^1_0(\Omega)^n$ for a.e. $t \in [0,T]$. 
	

	2. 
	For the function $P(x)$ defined by \eqref{eq:P_def} with
	positive definite matrices $P_1,\dots,P_{N_0} \in \mathbb{R}^{n \times n}$,
	we set
	\[
		V(z) := (z,Pz)_{L^2} \qquad \forall z \in L^2(\Omega)^n.
	\]
	Similarly to Theorem~\ref{thm:z_continuity},
	the Lyapunov function 
	$V(t) := V\big(z(t) \big)
	$ with the weak solution $z$ of the PDE \eqref{eq:coupled_PDE} is absolutely continuous on $[0,T]$ and 
	\[
	\frac{dV}{dt} (t)= 2 
	\left \langle 
	\frac{dz}{dt}(t) , Pz(t)
	\right\rangle\qquad \text{a.e.~} t \in [0,T].
	\]
	A routine calculation shows that 
		\begin{align}
	\nabla(Pz)(x) &= 
	\sum_{r=0}^{m} 
	\left( P_{p(k,r)} \otimes \nabla \alpha_{p(k,r)}(x) \right) z(x) 		\label{eq:nabla_P}\\
	&\quad + 
	({\bf P}_k(x) \otimes  I_m) \nabla z(x) 
	~~\text{a.e.~} x \in \Omega_k,~\forall k =1,\dots,N \notag 
	\end{align}
	for all $z \in H^1_0(\Omega)^n$.
	Since
	the Lebesgue measure of the boundary $\partial \Omega_k$ is zero 
	for every $k=1,\dots,N$,
	Lemma~\ref{lem:alpha_grad} and \eqref{eq:nabla_P} yield
	\begin{align*} 
	&2\Big(\big(A \otimes I_m) \nabla z(t), \nabla (Pz(t)\big)\Big)_{L^2} \\
	&=2\sum_{k=1}^N \int_{\Omega_k}
	\big((A \otimes I_m) \nabla z(x,t) \big)^{\top}
	\nabla \big(P(x) z(x,t)\big) dx \\
	&=
	\sum_{k=1}^N \int_{\Omega_k} 
	\begin{bmatrix}
	z(x,t) \\ \nabla z(x,t)
	\end{bmatrix}^{\top} \\
	&~~\quad \times
	\begin{bmatrix}
	0 & \sum_{r=0}^m \big((P_{p(k,r)} A) \otimes v_{p(k,r)}^{\top}\big) \\
	\star & (A^{\top}  {\bf P}_k (x)  + {\bf P}_k (x)  A) \otimes I_m
	\end{bmatrix} 
	\begin{bmatrix}
	z(x,t) \\ \nabla z(x,t)
	\end{bmatrix}
	dx.
	\end{align*}
	Moreover, \eqref{eq:B_bound2} guarantees that 
	for every $k=1,\dots,N$,
	there exists a measurable function $\Phi_k :\Omega_k \to \mathbb{R}^{n \times n}$ such that
	for a.e. $x \in \Omega_k$,
	\begin{subequations}
		\begin{align}
		B(x) - {\bf B}_k(x) &= \rho_k \Phi_k(x) 
		\label{eq:B_PWL}\\
		\|\Phi_k(x)\| &\leq 1.
		\label{eq:B_PWL_error}
		\end{align}
	\end{subequations}
	It follows from the condition 1) in Definition~\ref{defn:weak_solution} that
	for a.e. $t \in [0,T]$, 
	the Lyapunov function 
	$V(t)$ satisfies
	\begin{align}
	\label{eq:V_deriv_polytopic_case}
	&\frac{dV}{dt}(t) =
	\sum_{k=1}^{N}
	\int_{\Omega_k}
	\begin{bmatrix}
	z(x,t) \\ \nabla z(x,t)
	\end{bmatrix}^{\top} M_{k}(x)
	\begin{bmatrix}
	z(x,t) \\ \nabla z(x,t)
	\end{bmatrix}	dx.
	\end{align}
	Here we defined 
	$M_k: \Omega_k \to \mathbb{R}^{(m+1)n \times (m+1)n}$
	by
	\begin{align*}
	M_{k}(x) &:=
	\begin{bmatrix}
	M_k^{(1)}(x) \quad& M_k^{(2)}\\
	\star \quad& M_k^{(3)}(x)
	\end{bmatrix}
	\end{align*}
	with 
	\begin{align*}
	M_k^{(1)}(x) &:= \big({\bf B}_{k}(x) + \rho_k\Phi_k(x)\big)^{\top}{\bf P}_k(x)  \\
	&\qquad + {\bf P}_{k}(x)\big({\bf B}_{k}(x) + \rho_k\Phi_k(x)\big) \\
	M_k^{(2)} &:= -\sum_{r=0}^{m} (P_{p(k,r)}A) \otimes v_{p(k,r)}^{\top} \\
	M_k^{(3)}(x) &:= -(A^{\top}{\bf P}_k(x) \!+\! {\bf P}_k(x) A) \otimes I_m.
	\end{align*}
	
	3. 
	Let $\epsilon > 0 $.
	For every $k=1,\dots,N$, define 
	$G_k: \Omega_k \to \mathbb{R}^{(m+2)n \times (m+2)n}$ by
	\begin{align*}
	G_{k}(x) &:=
	\begin{bmatrix}
	G_k^{(1)}(x)  \quad&G_k^{(2)}(x)  \\
	\star \quad& G_k^{(3)}(x) 
	\end{bmatrix},
	\end{align*}
	where
	\begin{align*}
	G_k^{(1)}(x) &:= \Lambda  -  ({\bm \sigma}_k(x)  + \epsilon) I_n  \\
	&\qquad \qquad -  {\bf B}_{k}(x)^{\top}{\bf P}_k(x) 
	\!-\!  {\bf P}_{k}(x){\bf B}_{k}(x)  \\
	G_k^{(2)}(x) &:=	
	\begin{bmatrix}
	\rho_k {\bf P}_k (x) \quad &
	\sum_{r=0}^{m}(P_{p(k,r)}A)   \otimes  v_{p(k,r)}^{\top} 
	\end{bmatrix} \\
	G_k^{(3)}(x) &:=	
	\begin{bmatrix}
	{\bm \sigma}_k(x) I_n & 0 \\
	\star & (A^{\top}{\bf P}_k(x)  \!+\!  {\bf P}_k(x) A \!-\!  c^2\Lambda )  \!\otimes\!   I_m
	\end{bmatrix}\!.
	\end{align*}
	We now show 
	that if the LMIs \eqref{eq:LMI_1D} are feasible for
	all $k=1,\dots,N$ and for all $\ell=0,\dots,m$, then
	there exists $\epsilon >0$
	such that $G_k(x) \succeq 0$ for every $x \in \Omega_k$
	and every $k=1,\dots,N$.
	Define 
	\begin{align*}
	\Theta^{(1,1)}_{k,\ell}  &:= \Lambda - (\sigma_{k,\ell}+\epsilon) I_n- B_{p(k,\ell)}^{\top}\Upsilon_k - \Upsilon_k^{\top}B_{p(k,\ell)}  \\
	\Theta^{(1)}_{k,\ell} &:= 
	\begin{bmatrix}
	\Theta^{(1,1)}_{k,\ell}\quad &
	\Upsilon_k^{\top} - P_{p(k,\ell)} - B_{p(k,\ell)}^{\top}\Xi_k \\
	\star \quad& \Xi_k + \Xi_k^{\top}
	\end{bmatrix} \\
	\Theta^{(2)}_{k,\ell} &:= 
	\begin{bmatrix}
	\rho_k P_{p(k,\ell)} \quad& 
	\sum_{r = 0}^m (P_{p(k,r)}A) \otimes v_{p(k,r)}^{\top} \\
	0 \quad & 0 
	\end{bmatrix} \\
	\Theta^{(3)}_{k,\ell} &:= 
	\begin{bmatrix}
	\sigma_{k,\ell} I_n & 0 \\
	0 & (A^{\top}P_{p(k,\ell)} + P_{p(k,\ell)}A - c^2\Lambda) \otimes I_m
	\end{bmatrix}.
	\end{align*}
	For every $x \in \Omega_k$ and every $k=1,\dots,N$, we obtain
	\begin{align*}
	\Theta^{(1,1)}_{k}(x) &:=
	\Lambda \!-\! ({\bm \sigma}_{k}(x)\!+\!\epsilon) I_n \!-\! {\bf B}_{k}(x)^{\top}\Upsilon_k \!-\! \Upsilon_k^{\top} {\bf B}_{k}(x) \\
	\Theta^{(1)}_{k}(x) &:= 
	\begin{bmatrix}
	\Theta^{(1,1)}_{k}(x) \quad &
	\Upsilon_k^{\top} - {\bf P}_{k}(x) - {\bf B}_{k}(x)^{\top}\Xi_k \\
	\star \quad & \Xi_k + \Xi_k^{\top} 
	\end{bmatrix}	 \\
	&= \sum_{\ell=0}^{m}
	\alpha_{p(k,\ell)}(x) \Theta^{(1)}_{k,\ell} \\
	\Theta^{(2)}_{k}(x) &:=
	\begin{bmatrix} 
	G_k^{(2)}(x)\\ 0
	\end{bmatrix}
	= \sum_{\ell=0}^{m}
	\alpha_{p(k,\ell)}(x) \Theta^{(2)}_{k,\ell}\\
	\Theta^{(3)}_{k}(x) &:= G_k^{(3)}(x) = 
	\sum_{\ell=0}^{m}
	\alpha_{p(k,\ell)}(x) \Theta^{(3)}_{k,\ell}.
	\end{align*}
	If the LMIs \eqref{eq:LMI_1D} are feasible for
	all $k=1,\dots,N$ and for all $\ell=0,\dots,m$, then
	there exists $\epsilon > 0$ such that 
	\[
	\begin{bmatrix}
	\Theta^{(1)}_{k,\ell} \quad & \Theta^{(2)}_{k,\ell} \\
	\star \quad & \Theta^{(3)}_{k,\ell}
	\end{bmatrix} \succeq 0\qquad \forall k = 1,\dots,N,~
	\forall \ell = 0,\dots,m.
	\]
	Therefore, 
	\[
	\begin{bmatrix}
	\Theta^{(1)}_{k}(x) \quad & \Theta^{(2)}_{k}(x) \\
	\star \quad & \Theta^{(3)}_{k}(x)
	\end{bmatrix} =
	\sum_{\ell=0}^{m}
	\alpha_{p(k,\ell)}(x)
	\begin{bmatrix}
	\Theta^{(1)}_{k,\ell} \quad & \Theta^{(2)}_{k,\ell} \\
	\star \quad & \Theta^{(3)}_{k,\ell}
	\end{bmatrix} \succeq 0
	\]	 
	for all $x \in \Omega_k$ and for all 
	$k = 1,\dots,N$.
	Since $\Theta^{(3)}_{k}(x) \succ 0$ for all $x \in \Omega_k$
	provided that the inequalities \eqref{eq:LMI_1D} hold,
	the Schur complement formula shows that 
	\begin{align}
	0 &\preceq
	\Theta^{(1)}_{k}(x) - \Theta^{(2)}_{k}(x) \Theta^{(3)}_{k}(x)^{-1} 
	\Theta^{(2)}_{k}(x)^{\top} \notag \\
	&= 
	\begin{bmatrix}
	R_k(x) \quad &
	\Upsilon_k^{\top} - {\bf P}_{k}(x) - {\bf B}_{k}(x)^{\top}\Xi_k \\
	\star \quad & \Xi_k + \Xi_k^{\top} 
	\end{bmatrix}
	\label{eq:Rk_Up_Xi}
	\end{align}
	for all $x \in \Omega_k$ and for all
	$k = 1,\dots,N$,
	where $R_k(x)$ is defined by
	\begin{align*}
	R_k(x) &:= Q_k(x)
	- {\bf B}_{k}(x)^{\top}\Upsilon_k - \Upsilon_k^{\top} {\bf B}_{k}(x)\text{~~with}\\
	Q_k(x) &:=
	\Lambda - ({\bm \sigma}_{k}(x)+\epsilon) I_n -
	G_k^{(2)}(x)  G^{(3)}_{k}(x)^{-1} 
	G_k^{(2)}(x)^{\top}.
	\end{align*}
	Applying Lemma~\ref{lem:LMI_formula} to the inequality \eqref{eq:Rk_Up_Xi},
	we obtain
	\[
	Q_k(x) -  {\bf B}_{k}(x)^{\top}{\bf P}_{k}(x) - {\bf P}_{k}(x) {\bf B}_{k}(x)
	\succeq 0
	\]
	for all $x \in \Omega_k$ and for
	all $k = 1,\dots,N$. 
	Using the Schur complement formula again, we derive 
	$
	G_k(x) 
	\succeq 0 
	$
	for every $x \in \Omega_k$ and every $k=1,\dots,N$.

	4. 	
	By \eqref{eq:B_PWL_error}, $I - \Phi_k(x)^{\top} \Phi_k(x) \succeq 0$
	for a.e. $x \in \Omega_k$ 
	and every $k = 1,\dots,N$.
	Since $G_k(x) \succeq 0$,
	it follows that
	\begin{align*}
	0 &\preceq  \begin{bmatrix}
	I_n & 0 \\-\Phi_k(x) & 0 \\
	0 & I_n
	\end{bmatrix}^{\top}
	G_k(x)
	\begin{bmatrix}
	I_n & 0 \\-\Phi_k(x) & 0 \\
	0 & I_n
	\end{bmatrix} \\
	&= 
	-M_k(x) -
	\begin{bmatrix}
	{\bm \sigma}_k(x)\big(I_n - \Phi_k(x)^{\top} \Phi_k(x)\big) & 0 \\
	0 & 0
	\end{bmatrix}	\\
	&\hspace{45.5pt}- 
	\begin{bmatrix}
	\epsilon I_n & 0 \\
	0 & 0
	\end{bmatrix} 
	-
	\begin{bmatrix}
	-\Lambda & 0 \\
	0 & c^2\Lambda\otimes I_m
	\end{bmatrix}	\\
	&\preceq 	
	-M_k(x) 	
	-
	\begin{bmatrix}
	\epsilon I_n & 0 \\
	0 & 0
	\end{bmatrix}
	-
	\begin{bmatrix}
	-\Lambda & 0 \\
	0 & c^2\Lambda\otimes I_m
	\end{bmatrix}
	\end{align*}
	for a.e. $x \in \Omega_k$ and 
	for all $k=1,\dots,N$.
	Applying Corollary~\ref{coro:poincare},
	we obtain
	$
	\frac{dV}{dt}(t) \leq -\epsilon \|z(t)\|_{L^2}^2
	$ by \eqref{eq:V_deriv_polytopic_case}.
	Thus, the coupled parabolic system is exponentially stable
	from the same argument in  3. of the proof of Theorem~\ref{thm:stability}.
	This completes the proof.
\end{proof}

\begin{remark}[Complexity of  LMIs in  Theorem~\ref{thm:piecewise_linear}]
	In Theorem~\ref{thm:piecewise_linear},
	the total number $N_0$ of the vertices satisfies $N_0 \leq (m+1) N$.
	Therefore, there are $O(n^2N_0) = O(n^2mN)$ variables in the matrices $P_1,\dots,P_{N_0}$.
	The number of variables in the diagonal matrix $\Lambda$ is $O(n)$, and the number of scalar variables
	$\sigma_{k,\ell}$ $(k=1,\dots,N,~\ell=0,\dots,m)$ is $O(mN)$. Hence
	the LMIs of Theorem~\ref{thm:piecewise_linear} contain $O(n^2mN)$ variables in total.
	If $N=2^m$, then the number of variables satisfies
	$O(n^2m2^m)$.
	Let us next consider the case where 
	all the intersections of the $m$-simplices are their facets, i.e., $(m-1)$-simplices.  Then
	$N_0 \leq  m + N$, and hence
	the LMIs of Theorem~\ref{thm:piecewise_linear} has
	$O(n^2m+(n^2+m)N)$ variables. If $N=2^m$, then the worst-case number is given by
	$O\big((n^2+m)2^m\big)$.
\end{remark}
\section{Examples}
\subsection{1-D case}
Let $b \geq 0$ and 
consider the coupled 1-D parabolic system \eqref{eq:coupled_PDE} with
$\Omega = (0,1)$ and 
\begin{equation}
\label{eq:A_B_example}
A = 
\begin{bmatrix}
1 & 0.1 \\
0.5 & 1
\end{bmatrix},\quad
B(x) = 
2
\begin{bmatrix}
\sin(2\pi x)  & \tan(x) \\
\cos(\pi x) & 2x 
\end{bmatrix} + b I_2.
\end{equation}
Since the coefficient matrix $B(x)$ is not polynomial,
the techniques developed in 
the previous studies \cite{Valmorbida2015CDC, Valmorbida2016, Gahlawat2017, Gahlawat2017CDC, Solomon2015}
cannot be applied to this system.
To use the obtained results, 
we divide $\Omega$ into $N=100$ intervals 
\[
\Omega_k := 
\begin{cases}
\left(\frac{k-1}{N},\frac{k}{N}\right]  & \text{if~}k=1,\dots,N-1 \\
\left(\frac{N-1}{N},1\right) & \text{if~}k = N. 
\end{cases}
\]
The constant $B_k$ in \eqref{eq:B_bound} for Theorem~\ref{thm:stability} is chosen as
\[
B_k=B\left( \frac{2k-1}{2N}\right)\quad \forall k =1 ,\dots,N,
\]
which is the value of $B$ at the center of the interval $\Omega_k$.
Since 
the vertices of $\overline{\Omega}_k$ are 
$\xi_{p(k,0)} = (k-1)/N$ and $\xi_{p(k,1)} = k/N$,
the constant $B_{p(k,\ell)}$ in \eqref{eq:B_bound2} is given by
\[
B_{p(k,0)}=B\left(\frac{k-1}{N}\right),~B_{p(k,1)}=B\left(\frac{k}{N}\right)\quad \forall k =1 ,\dots,N.
\]
We numerically compute
the bound $\rho_k$ in \eqref{eq:B_bound} 
for Theorem~\ref{thm:stability} based on the following approximation:
\begin{align*}
\rho_k = \max_{x \in \overline{\Omega}_k}
\left\| 
B(x) -  B_{k}  
\right\| \approx
\max_{x \in \Omega_k^{\rm a}}
\left\| 
B(x) - B_{k}  
\right\| \quad 
\forall k = 1,\dots,N,
\end{align*}
where 
\[
\Omega_k^{\rm a} :=
\left\{
\frac{k-1}{N},
\frac{k-1}{N} + \frac{1}{20N},\dots,
\frac{k-1}{N} + \frac{2}{20N},\dots,
\frac{k}{N}
\right\}.
\]
The bound $\rho_k$
in \eqref{eq:B_bound2} for Theorem~\ref{thm:piecewise_linear} is computed in the same brute force way.
The constants $c$ in Poincar\'e-Friedrichs' inequality \eqref{eq:poincare_ineq}
and $v_{p(k,\ell)}$ in the LMI \eqref{eq:LMI_1D} are given by 
$c = 1/\pi$, $v_{p(k,0)} = -N$, and $v_{p(k,1)} = N$ ($k=1,\dots,N$), respectively.
Using finite differences with 1000 
uniformly distributed spatial points, we find that the approximated parabolic PDE
is stable if $b \leq 8.35$. The LMIs  in Theorems~\ref{thm:stability} and Theorem~\ref{thm:piecewise_linear} 
are feasible
for $b \leq 6.66$ and $b \leq 6.84$, respectively.  
From this example, we observe the effectiveness
of Lyapunov functions that depend on the space variable in a
piecewise linear fashion.


\subsection{3-D case}
Next, we illustrates the advantage of Lyapunov functions with constant $P$,
which allow us to analyze the stability of parabolic PDEs on a general set $\Omega$.
We consider
the coupled 3-D parabolic system \eqref{eq:coupled_PDE} with the unit ball
$\Omega = \{(x_1,x_2,x_3) \in \mathbb{R}^3:x_1^2+x_2^2+x_3^2 < 1\}$ and
the coefficient matrices
$A$ in \eqref{eq:A_B_example} and
\[
B(x_1,x_2,x_3) = 2
\begin{bmatrix}
\sin\big(2\pi (x_1+x_2)\big)  & \tan(x_3) \\
\cos \big(\pi (x_2+x_3)\big) & 2x_1 
\end{bmatrix} + b I_2,
\]
where $b \geq 0$.
The previous studies \cite{Valmorbida2015CDC,Solomon2015,Fridman2016}
for multi-dimensional PDEs focus on the case where $\Omega$ is a box.
Although balls are also basic sets, relatively little work has been done on
stability analysis for parabolic PDEs on balls.
Using the fact on the Rayleigh quotient for the Laplace operator (See, e.g., Theorem 2 in Sec.~6.5.1 of \cite{Evans1998}), 
we choose
the constant $c$ in Poincar\'e-Friedrichs' inequality \eqref{eq:poincare_ineq} as $c = 1/\pi$.
We divide $\Omega = 
\big\{
(r,\theta,\phi): r \in [0,1),~\theta \in [0,\pi], \phi \in [0,2\pi) 
\big\}$ by uniformly splitting
the intervals $[0,1)$, $[0,\pi]$, and $[0,2\pi)$ into $N \in \{5,10,15,20,25,30\}$ segments, respectively.
As in the 1-D case above,
the constant $B_k$ in \eqref{eq:B_bound} is set to
the value of $B$ at the center of each segment, and 
the bound $\rho_k$ in \eqref{eq:B_bound}
is numerically computed with a sufficiently fine grid.
Table~\ref{tab:b_data} describes the maximum $b \geq 0$ for which the LMIs in Theorem~\ref{thm:stability}
are feasible.  
This table shows that a large $N$ is required to obtain less conservative results.

\begin{table}[htb]
	\centering
	\caption{Maximum $b \geq 0$ for which LMIs in \eqref{eq:stability_LMI}
		are feasible.}
	\label{tab:b_data}
	\begin{tabular}{c|cccccc} 
		$N$ & 5 & 10 & 15 & 20 & 25 & 30\\ \hline
		$b$ & Infeasible & 0.14 & 0.84 & 1.07 & 1.67 & 2.07
	\end{tabular}
\end{table}

\section{Conclusion}
We have studied the stability analysis of coupled parabolic systems with
spatially varying coefficients.
Employing the gridding method developed for systems with
aperiodic sampling and time-varying delays, 
we have obtained LMI-based sufficient conditions for exponential stability.
Future work will focus on extending this gridding method to 
various classes of distributed parameter systems.
Another interesting direction  for future research would be to make 
stability analysis more accurate by
using integration operators with kernels for Lyapunov functions as in 
\cite{Gahlawat2017, Gahlawat2017CDC}.
If $B$ is a polynomial and $\Omega$ is a convex polytope, then sum-of-squares-based analysis through
P\`olya's theorem and Handelman representations is expected to be a less conservative 
approach.

\section{Proof of Existence and Uniqueness of Weak Solution}
For the PDE \eqref{eq:coupled_PDE},
define the function $a: H^1_0(\Omega)^n \times H^1_0(\Omega)^n \to \mathbb{R}$ by
\[
a(\zeta,v) := 
\big(
(A \otimes I_m) \nabla \zeta, \nabla v
\big)_{L^2} - (B\zeta,v)_{L^2}.
\]
We first obtain the following estimates on the function $a$:
\begin{lemma}
	Under Assumption~\ref{assum:A},
	there exist constants $C_1, C_2>0$, depending only on $\Omega$ and
	the coefficients $A,B$,  such that 
	\begin{subequations}
		\begin{align}
		|a(\zeta,v)| &\leq C_1 \|\zeta\|_{H^1_0} \cdot \|v\|_{H^1_0} \qquad \forall \zeta,v \in H^1_0(\Omega)^n
		\label{eq:a_bound}\\
		\alpha \|\zeta\|^2_{H^1_0} &\leq a(\zeta,\zeta) + C_2\|\zeta\|_{L^2}^2\qquad \forall \zeta \in H^1_0(\Omega)^n.
		\label{eq:a_coercive}
		\end{align}
	\end{subequations}
\end{lemma}
\begin{proof}
	We obtain the first inequality \eqref{eq:a_bound} by
	\begin{align*}
	|a(\zeta,v)| &\leq 
	\big|
	\big((A\otimes I_m)\nabla \zeta,\nabla v \big)_{L^2}
	\big| +
	\big|
	(B\zeta,v )_{L^2}
	\big|	\\	 
	&\leq 
	\|A \otimes I_m \| \!\cdot \! \|\nabla \zeta\|_{L^2} \!\cdot \! \| \nabla v\|_{L^2} + 
	\|B\|_{L^\infty} \!\cdot \! \|\zeta\|_{L^2} \!\cdot \! \|v\|_{L^2} \\
	&\leq C_1 \|\zeta\|_{H^1_0} \cdot \|v\|_{H^1_0}\qquad \text{for some $C_1 >0$.}
	\end{align*}
	
	To obtain the second inequality \eqref{eq:a_coercive}, 
	we see from Assumption~\ref{assum:A} 
	that for every $\zeta = \begin{bmatrix}
	\zeta_1 & \dots & \zeta_n
	\end{bmatrix}^\top \in H^1_0(\Omega)^n,
	$
	\begin{align*}
	\big(
	(A \otimes I_m) \nabla \zeta, \nabla \zeta
	\big)_{L^2}
	&=
	\int_\Omega 
	\sum_{\ell=1}^m
	\begin{bmatrix}
	\frac{\partial \zeta_1}{\partial x_\ell}(x) \\ \vdots \\
	\frac{\partial \zeta_n}{\partial x_\ell}(x) 
	\end{bmatrix}^{\top} A
	\begin{bmatrix}
	\frac{\partial \zeta_1}{\partial x_\ell}(x)  \\ \vdots \\
	\frac{\partial \zeta_n}{\partial x_\ell}(x) 
	\end{bmatrix} dx \\
	&\geq 
	\alpha 	\int_\Omega 
	\sum_{\ell=1}^m
	\left\|
	\begin{bmatrix}
	\frac{\partial \zeta_1}{\partial x_\ell}(x)  \\ \vdots \\
	\frac{\partial \zeta_n}{\partial x_\ell}(x) 
	\end{bmatrix}
	\right\|^2 dx
	= \alpha \|\nabla \zeta\|^2_{L^2}.
	\end{align*} 
	Therefore, 
	\begin{align*}
	\alpha \|\nabla \zeta\|^2_{L^2} 
	&\leq 	(
	\big(A \otimes I_m) \nabla \zeta, \nabla \zeta
	\big)_{L^2} \\
	&\leq a(\zeta,\zeta) + \|B\|_{L^\infty} \cdot \|\zeta\|_{L^2}^2.
	\end{align*}
	Thus, 
	\[
	\alpha \|\zeta\|^2_{H^1_0} \leq a(\zeta,\zeta) + (\alpha + \|B\|_{L^\infty} ) \|\zeta\|_{L^2}^2.
	\]
	This completes the proof.
\end{proof}

Let us next apply Galerkin's method.
Let $\{e_k:k \in \mathbb{N}\}$ be an orthonormal basis of $L^2(\Omega)$ and
an orthogonal basis of $H^1_0(\Omega)$. 
Define $\{w_k:k \in \mathbb{N}\} \subset L^2(\Omega)^n$ by
\[
w_1 := 
\begin{bmatrix}
e_1 \\ 0 \\ \vdots \\ 0
\end{bmatrix},~
w_2 := 
\begin{bmatrix}
0 \\ e_1 \\ \vdots \\ 0
\end{bmatrix},\dots
w_n := 
\begin{bmatrix}
0 \\ \vdots \\ 0 \\ e_1
\end{bmatrix},~
w_{n+1} :=
\begin{bmatrix}
e_2 \\ 0 \\ \vdots  \\ 0
\end{bmatrix},\dots.
\] 
Then $\{w_k:k \in \mathbb{N}\} $ is an 
an orthonormal basis of $L^2(\Omega)^n$ and
an orthogonal basis of $H^1_0(\Omega)^n$. 
Define a finite-dimensional subspace $E_N := \{ 
w_1,\dots,w_N
\} \subset H^1_0 (\Omega)^n$.

We now prove that there uniquely exist absolutely continuous functions $\psi_N^1,\dots,\psi_N^N:[0,T] \to \mathbb{R}$ such that 
the function $z_N$ defined by
\begin{equation}
\label{eq:approximate_solution}
z_N(t) := \sum_{k=1}^N \psi_N^k(t)w_k 
\end{equation}
satisfies 
\begin{equation}
\label{eq:app_space}
z_N,\frac{dz_N}{dt} \in L^2(0,T;E_N)
\end{equation}
and
\begin{align}
\label{eq:app_PDE}
\begin{cases}
\displaystyle
\left(\frac{dz_N}{dt}(t), w_k\right)_{L^2} + a(z_N(t),w_k) = 0 \\
\hspace{38pt}\text{a.e.~}t \in [0,T],~\forall k =1,\dots,N\\
\psi_N^k(0) = (z^0,w_k)_{L^2} \qquad \forall k =1,\dots,N.
\end{cases}
\end{align}

\begin{lemma}
	\label{lem:approximate_solution}
	For each $N \in \mathbb{N}$, 
	there exists a unique function $z_N$ of the form \eqref{eq:approximate_solution}
	with absolutely continuous coefficients $\psi_N^1,\dots,\psi_N^N$
	such that \eqref{eq:app_space} and \eqref{eq:app_PDE} hold.
\end{lemma}
\begin{proof}
	By definition, $z_N(t)$ of the form \eqref{eq:approximate_solution}
	satisfies \eqref{eq:app_space} if and only if
	\begin{equation}
	\label{eq:app_coeff_space}
	\psi_N^k,\frac{d\psi_N^k}{dt} \in L^2(0,T)\qquad \forall k=1,\dots,N.
	\end{equation}
	Moreover,
	we obtain
	\[
	\left(\frac{dz_N}{dt}(t),w_k\right)_{L^2} = \frac{d\psi_N^{k}}{dt}(t).
	\]
	and
	\begin{align*}
	a(z_N(t),w_k) &= 
	\sum_{\ell=1}^N
	\big(
	(A \otimes I_m) \psi_N^{\ell} (t) \nabla w_\ell,\nabla w_k
	\big)_{L^2} \\ &\qquad - 
	(B \psi_N^{\ell} (t)w_\ell,w_k)_{L^2} \\
	&=
	\sum_{\ell=1}^N  a(w_\ell,w_k)\psi_N^\ell (t).
	\end{align*}
	Hence the first equation in \eqref{eq:app_PDE}  holds if and only if
	\[
	\frac{d
		\psi_N^{k}}{dt} (t)+ \sum_{\ell=1}^N a(w_\ell,w_k) \psi_N^\ell (t) = 0
	\]
	for a.e. $t \in [0,T]$ and every $k = 1,\dots,N$,
	which is equivalent to 
	\begin{equation}
	\label{eq:ODE}
	\frac{d\psi_N}{dt}(t) +A_N \psi_N(t) = 0\qquad
	\text{a.e.~}t \in [0,T], 
	\end{equation}
	where
	\[
	\psi_N := 
	\begin{bmatrix}
	\psi_N^1 \\ \vdots \\ \psi_N^N
	\end{bmatrix},~
	A_N := 
	\begin{bmatrix}
	a(w_1,w_1) & \cdots & a(w_N,w_1) \\
	\vdots & & \vdots \\
	a(w_N,w_1) & \cdots & a(w_N,w_N) 
	\end{bmatrix}.
	\]
	The ordinary differential equation \eqref{eq:ODE} with initial data given by
	the second equation in \eqref{eq:app_PDE} has a continuously differentiable solution,
	which satisfies \eqref{eq:app_coeff_space}.
	Thus, there exists a function $z_N$ of the form \eqref{eq:approximate_solution}
	with absolutely continuous coefficients $\psi_N^1,\dots,\psi_N^N$
	such that \eqref{eq:app_space} and \eqref{eq:app_PDE} is satisfied.
	
	To prove the uniqueness, it suffices to show that 
	if $z_N(0) = 0$, then $z_N(t) = 0$ for every $t \in [0,T]$.
	This will be proved in Lemma~\ref{lem:energy_estimate} below.
\end{proof}

We next evaluate the energy of approximate solutions $z_N$.
\begin{lemma}
	\label{lem:energy_estimate}
	Under Assumption~\ref{assum:A},
	there exists a constant $C >0$, depending only on $T$, $\Omega$, and
	the coefficients $A,B$, such that 
	for every $N \in \mathbb{N}$, the approximate solution $z_N$ constructed in 
	Lemma~\ref{lem:approximate_solution} satisfies
	\begin{align*}
	\|z_N\|_{L^{\infty}(0,T;L^2)}  &+ \|z_N\|_{L^2(0,T; H^1_0)} \\
	&+ 
	\left\|\frac{dz_N}{dt}\right\|_{L^2(0,T;H^{-1})} \leq C \|z^0\|_{L^2}.
	\end{align*}
\end{lemma}
\begin{proof}
	Since $z_N(t) \in E_M$, it
	follows from \eqref{eq:app_PDE} that
	\[
	\left(\frac{dz_N}{dt}(t), z_N(t)\right)_{L^2} + a(z_N(t),z_N(t)) = 0 \qquad \text{a.e.~} t \in [0,T].
	\]
	Therefore, \eqref{eq:a_coercive} yields
	\begin{equation}
	\label{eq:alpha_H10}
	\frac{1}{2} \frac{d}{dt} \|z_N(t)\|_{L^2}^2 + 
	\alpha \|z_N(t)\|^2_{H^1_0} \leq C_2 \|z_N(t)\|_{L^2}^2\quad \text{a.e.~} t \in [0,T],
	\end{equation}
	which implies that 
	\[
	\frac{d}{dt} \|z_N(t)\|_{L^2}^2 \leq 2C_2 \|z_N(t)\|_{L^2}^2\qquad \text{a.e.~} t \in [0,T].
	\]
	Since the function $t \to \|z_N(t)\|_{L^2}^2$ is absolutely continuous, it follows from
	Gronwall's inequality that  
	\begin{equation}
	\label{eq:z_N_bound}
	\|z_N(t)\|_{L^2}^2 \leq e^{-2C_2t} \|z_N(0)\|_{L^2}^2 
	\qquad \forall t\in[0,T].
	\end{equation}
	Since $\|z_N(0)\|_{L^2} \leq \|z^0\|_{L^2}$,
	there exists $C_3>0$ such that 
	\[
	\|z_N\|_{L^{\infty}(0,T;L^2)} \leq C_3 \|z^0\|_{L^2}.
	\]
	Moreover, \eqref{eq:z_N_bound} shows that 
	$z_N(t) = 0$ for every $t\in[0,T]$ if $z_N(0)=0$.
	Therefore, the approximate solution $z_N$ constructed in 
	Lemma~\ref{lem:approximate_solution} is unique.
	
	By \eqref{eq:alpha_H10}, we also derive
	\[
	\alpha \|z_N(t)\|_{H^1_0}^2 \leq C_2 \|z_N(t)\|_{L^2}^2\qquad \text{a.e.~} t \in [0,T],
	\]
	and hence there exists $C_4 > 0 $ such that 
	\begin{align*}
	\|z_N\|^2_{L^2(0,T; H^1_0)} &= 
	\int^T_0 \|z_N(t)\|_{H^1_0}^2 dt \\&\leq
	\frac{C_2}{\alpha}  \int^T_0 \|z_N(t)\|_{L^2}^2 dt \\
	&\leq 
	C_4 \|z^0\|^2_{L^2}.
	\end{align*}	
	Fix $v \in H^1_0(\Omega)$ with $\|v\|_{H^1_0} \leq 1$.
	We can decompose $v$ as $v = v^1+v^2$ with $v^1 \in E_N$ and 
	$(v^2,w_k)_{L^2} = 0$ for every $k= 1,\dots,N$.
	Then $\|v^1\|_{H^1_0} \leq \|v\|_{H^1_0} \leq 1$ and 
	\[
	\left(\frac{dz_N}{dt}(t),v^1\right)_{L^2} + a(z_N(t),v^1) = 0\qquad \text{a.e.~}t \in [0,T].
	\]
	Using \eqref{eq:a_bound}, we therefore obtain
	\begin{align*}
	\left|
	\left\langle \frac{dz_N}{dt}(t), v \right\rangle
	\right| &=
	\left|
	\left(\frac{dz_N}{dt}(t), v \right)_{L^2}
	\right| 
	=
	\left|
	\left(\frac{dz_N}{dt}(t), v^1 \right)_{L^2}
	\right| \\
	&=
	|a(z_N(t),v^1)| \\
	&\leq C_1 \|z_N(t)\|_{H^1_0} \cdot \|v^1\|_{H^1_0} \\
	& \leq C_1 \|z_N(t)\|_{H^1_0}
	\qquad \text{a.e.~}t \in [0,T].
	\end{align*}
	This implies that 
	\[
	\left\|\frac{dz_N}{dt}(t)\right\|_{H^{-1}} \leq C_1\|z_N(t)\|_{H^1_0}
	\qquad \text{a.e.~}t \in [0,T].
	\]
	Thus 
	\begin{align*}
	\left\|\frac{dz_N}{dt}(t)\right\|_{L^2(0,T;H^{-1})}^2 &= 
	\int^T_0 \left\|\frac{dz_N}{dt}(t)\right\|_{H^{-1}}^2 dt \\&\leq
	\int^T_0 C_1^2  \|z_N(t)\|_{H^1_0}^2 dt \\
	&= C_1^2\|z_N\|^2_{L^2(0,T; H^1_0)} \\ &\leq C_1^2 C_4 \|z^0\|_{L^2}^2.
	\end{align*}
	This completes the proof.
\end{proof}

Since $H^{-1}(\Omega)^n$ is the dual of $H^1_0(\Omega)^n$ (with respect to the pivot space $L^2(\Omega)^n$),
we can identify the dual space of $L^2(0,T;H^1_0(\Omega)^n)$ with $L^2(0,T;H^{-1}(\Omega)^n)$ by
Theorem 6.30 of \cite{Hunter2014}.
Since $H^1_0(\Omega)^n$ is reflexive, it follows that 
the dual space of $L^2(0,T;H^{-1}(\Omega)^n)$ can be also identified with $L^2(0,T;H^1_0(\Omega)^n)$.
Using these facts, we show the existence of weak solutions.

\begin{theorem}
	Under Assumption~\ref{assum:A},
	a subsequence of the approximate solutions $\{z_N:N\in \mathbb{N} \}$  constructed in 
	Lemma \ref{lem:approximate_solution} converges weakly to a weak solution 
	of \eqref{eq:coupled_PDE}.
\end{theorem}
\begin{proof}
	Lemma \ref{lem:energy_estimate} shows that 
	the approximate solutions $\{z_N:N\in \mathbb{N} \}$ are bounded in $L^2(0,T;H^1_0(\Omega)^n)$ and that 
	their time-derivatives 
	$\left\{\frac{dz_N}{dt}:N\in \mathbb{N} \right\}$ are bounded in $L^2(0,T;H^{-1}(\Omega)^n)$.
	Therefore, by the Banach-Alaoglu thoerem and Problem 7.5.4 of \cite{Evans1998},
	we can extract a subsequence, which is still denoted by $\{z_N:N \in \mathbb{N} \}$, such that 
	the following weak convergences hold:
	\begin{align*}
	z_N &\rightharpoonup z \text{~in~} 
	L^2(0,T;H^1_0(\Omega)^n)\\
	\frac{dz_N}{dt} &\rightharpoonup \frac{dz}{dt}\text{~in~} 
	L^2(0,T;H^{-1}(\Omega)^n).
	\end{align*}
	Fix $N,M \in \mathbb{N}$ with $N \geq M$ and $\phi \in C_0^{\infty}(0,T)$, and take $w \in E_M$.
	Here $C_0^{\infty}(0,T)$ means the space of functions with
	continuous derivatives of all orders and compact support in $(0,T)$.
	By \eqref{eq:app_PDE}, we find that 
	\[
	\left(\frac{dz_N}{dt}(t),\phi(t)w\right)_{L^2} + a(z_N(t),\phi(t)w) = 0\qquad \text{a.e.~} t \in [0,T].
	\]
	Integrating it with respect to $t$, we obtain
	\[
	\int^T_0 \left(\frac{dz_N}{dt}(t),\phi(t)w\right)_{L^2} + a(z_N(t),\phi(t)w)  dt = 0.
	\]
	Since $\frac{dz_N}{dt} \rightharpoonup \frac{dz}{dt}$ in $L^2(0,T;H^{-1}(\Omega)^n)$, 
	it follows that 
	\begin{equation}
	\label{eq:weak_conv_deri}
	\int^T_0  \left\langle \frac{dz_N}{dt}(t),\phi(t)w \right\rangle dt \to 
	\int^T_0  \left\langle \frac{dz}{dt}(t),\phi(t)w \right\rangle  dt.
	\end{equation}
	Define the linear operator $Q$ on $L^2(0,T;H^1_0(\Omega)^n)$ by
	\[
	(Q\zeta)(t) := a(\zeta(t),\phi(t) w) \qquad \forall \zeta \in L^2(0,T;H^1_0(\Omega)^n).
	\]
	Using \eqref{eq:a_bound}, we obtain
	\begin{align*}
	\int^T_0 | 
	a(\zeta(t),\phi(t)w)
	|^2 dt &\leq 
	\int^T_0
	C_1^2 \|\zeta(t)\|_{H^1_0}^2 \cdot |\phi(t)|^2 \cdot \|w\|_{H^1_0}^2 dt \\
	&\hspace{-20pt}\leq
	C_1^2 \max_{0\leq t\leq T}|\phi(t)|^2 \cdot \|w\|_{H^1_0}^2 \cdot  \|\zeta\|_{L^2(0,T;H^1_0)}^2
	\end{align*}
	for all $\zeta \in L^2(0,T;H^1_0(\Omega)^n)$.
	Therefore, $Q$ is a bounded operator from $L^2(0,T;H^1_0(\Omega)^n)$
	to $L^2(0,T)$. Since $z_N \rightharpoonup z$ in $L^2(0,T;H^1_0(\Omega)^n)$,
	it follows that $Qz_N \rightharpoonup Qz$ in $L^2(0,T)$.
	In fact, choose $g \in L^2(0,T)$ arbitrarily and define 
	$h \in L^2(0,T;H^1_0(\Omega)^n)'$ by
	$h(\zeta) := (Q\zeta,g)_{L^2}$. Then 
	\[
	(Qz_N,g)_{L^2} = 
	h(z_N) \to h(z) =  (Qz,g)_{L^2}\qquad (N \to \infty).
	\]
	Thus we have that 
	for every $g \in L^2(0,T)$, 
	\[
	\int^T_0 \big(Q z_N(t) - Qz(t)\big) g(t) dt \to 0.
	\]
	In particular, if we set $g \equiv 1$, then we obtain
	\begin{equation}
	\label{eq:weak_conv}
	\int^T_0 a(z_N(t),\phi(t)w) dt \to  \int^T_0 a(z(t),\phi(t)w) dt.
	\end{equation}
	By \eqref{eq:weak_conv_deri} and \eqref{eq:weak_conv},
	\[
	\int^T_0 \phi(t) 
	\left(
	\left\langle \frac{dz}{dt}(t),v \right\rangle + a(z(t),w) 
	\right) dt = 0.
	\]
	This yields 
	\begin{align}
	&\left\langle \frac{dz}{dt}(t),w \right\rangle + a(z(t),w) = 0 \notag \\
	&\qquad \qquad \text{a.e.~} t \in [0,T],~\!
	\forall w \in E_M,~\!\forall M \in \mathbb{N}.	\label{eq:PDE_weak_case}
	\end{align}
	Since $\bigcup_{M \in \mathbb{N}} E_M$ is dense in $H^1_0(\Omega)^n$,
	\[
	\left\langle \frac{dz}{dt}(t),v \right\rangle + a(z(t),v) = 0  \qquad \text{a.e.~} t \in [0,T],~
	\forall v \in H^1_0(\Omega)^n.
	\]
	
	Let us next show that $z$ satisfies the initial condition $z(0)=z^0$.
	To that purpose, fix $\phi \in C^{\infty} [0,T]$ with $\phi(0)=1$ and $\phi(T) = 0$.
	Let $N,M \in \mathbb{N}$ with $N \geq M$.
	Using the integration by parts formula (see, e.g., Theorem 6.42 in \cite{Hunter2014}),
	we obtain
	\[
	\int^T_0 \!
	\left\langle \frac{dz}{dt}(t),\phi(t)w \right\rangle  dt =
	-( z(0), w )_{L^2}
	-
	\int^T_0 \frac{d\phi}{dt} (t)(z(t),w)_{L^2} dt
	\]
	for all $w \in E_M$.
	Hence, \eqref{eq:PDE_weak_case} yields
	\[
	(z(0), w )_{L^2} = 
	\int^T_0
	\phi(t) a(z(t),w) dt -\int^T_0 \frac{d\phi}{dt}(t) (z(t),w)_{L^2} dt .
	\]
	On the other hand,
	the approximate solution $z_N$ satisfies
	\[
	(z^0, w )_{L^2} = 
	\int^T_0
	\phi(t) a(z_N(t),w) dt -\int^T_0 \frac{d\phi}{dt}(t) (z_N(t),w)_{L^2} dt .
	\]
	Similarly to \eqref{eq:weak_conv_deri} and \eqref{eq:weak_conv}, we obtain
	\begin{align*}
	\int^T_0 \phi(t) a(z_N(t),w) dt &\to \int^T_0
	\phi (t)a(z(t),w) dt \\
	\int^T_0 \frac{d\phi}{dt}(t) (z_N(t),w)_{L^2} dt  &\to \int^T_0 \frac{d\phi}{dt}(t) (z(t),w)_{L^2} dt,
	\end{align*}
	which yields $(z(0), w )_{L^2} =(z^0, w )_{L^2} $ for every $w \in E_M$, $M \in \mathbb{N}$ and hence
	for every $w \in H^1_0(\Omega)^n$. Thus $z(0) = z^0$.
\end{proof} 

Finally, we show the uniqueness of weak solutions.
\begin{theorem}
	Under Assumption~\ref{assum:A},
	a weak solution 
	of \eqref{eq:coupled_PDE} is unique.
\end{theorem}
\begin{proof}
	If $z_1$ and $z_2$ are weak solutions of \eqref{eq:coupled_PDE}, then
	$z = z_1 - z_2$ is also a weak solution of \eqref{eq:coupled_PDE} with $z^0 = 0$.
	It suffices to show that $z \equiv 0$ is the only weak solution of 
	\eqref{eq:coupled_PDE} with $z^0  = 0$.
	
	Let $z$ be a weak solution of \eqref{eq:coupled_PDE} with $z^0  = 0$.
	By the condition~1) in Definition~\ref{defn:weak_solution} with $v = z(t)$, we obtain
	\[
	\left\langle
	\frac{dz}{dt}(t),z(t)
	\right\rangle + a\big(z(t),z(t)\big) = 0\qquad \text{a.e.~} t \in [0,T].
	\]
	By Theorem~\ref{thm:z_continuity} and \eqref{eq:a_coercive}, 
	\[
	\frac{d}{dt} \|z(t)\|_{L^2}^2 \leq 2C_2 \|z(t)\|_{L^2}^2 \qquad \text{a.e.~} t \in [0,T].
	\]
	Gronwall's inequality shows that 
	\[
	\|z(t)\|_{L^2}^2  \leq e^{2C_2 t}  \|z^0 \|_{L^2}^2 =0\qquad \forall t \in [0,T].
	\]
	Thus $z \equiv 0$.
\end{proof}
	
	\ifCLASSOPTIONcaptionsoff
	\newpage
	\fi


\end{document}